\documentclass{amsart}
\usepackage{amsthm}
\usepackage{amsmath}
\usepackage{amssymb}
\usepackage{euscript}
\usepackage{graphicx}
\usepackage{pgf,tikz}
\usetikzlibrary{arrows}
\usepackage[singlelinecheck=false,justification=justified]{caption}

\usepackage{multicol}

\usepackage{pgfplots}
\usetikzlibrary{decorations.pathreplacing}

\usetikzlibrary{external} 
\usepackage{adjustbox}
\definecolor{color0}{HTML}{000000}
\definecolor{color1}{HTML}{000000}
\definecolor{bgColor}{HTML}{ffffff}
\usetikzlibrary {arrows.meta,bending}

\pgfplotsset{compat=1.18}

\newtheorem{theorem}{Theorem}[section]
\newtheorem{lemma}[theorem]{Lemma}

\newtheorem{definition}[theorem]{Definition}

\newtheorem{question}[theorem]{Question}

\numberwithin{equation}{section}

\begin{document}




\title[Graphoids]{Shadowing and inverse shadowing in spaces retractable onto graphs}

\author[J. Meddaugh]{Jonathan Meddaugh}
\address[J. Meddaugh]{Department of Mathematics, Baylor University, Waco TX, 76798}
\email[J. Meddaugh]{Jonathan\_Meddaugh@baylor.edu}

\author[E. Stephens]{Elyssa Stephens}
\address[E. Stephens]{Department of Mathematics, Baylor University, Waco TX, 76798}
\email[E. Stephens]{Ellie\_Stephens2@baylor.edu}

\thanks{The first author was supported by a grant from the Simons Foundation (960812, JM)}

\date{}

\begin{abstract}

In this paper, we show that shadowing and inverse shadowing are generic properties for spaces in which there exists an $\varepsilon$-retraction onto a graph for any $\varepsilon>0$.  

\end{abstract}
\maketitle

\section{Introduction}

It is often the case, in the analysis of a dynamical system, that orbits may only be approximated, resulting in apparent behaviors that exist only as artifacts of the approximation process, rather than as actual behaviors for the system. This is easily seen, even for iteration of simple, well-behaved maps in which rounding errors can rapidly compound. Dynamical systems for which approximate orbits are well modeled by true orbits from the system are said to have the \emph{shadowing property}. The study of this property began with the work of Bowen \cite{bowenshadowintro}. Since then, this property has been studied in a variety of contexts, including in many results in ergodic theory \cite{ergodiconcompact}, the study of $\omega$-limit sets \cite{Shadowingomegajulia}, and a complete classification of the subshifts of finite type \cite{Walters1978}.

Since maps with shadowing have many desirable properties, it is worth knowing whether a given map has shadowing. For Axiom A diffeomorphisms, this is completely classified by a transversality condition \cite{robinson1977stability}. In shift spaces, it is well-known that the subshifts with the shadowing property are precisely the subshifts of finite type \cite{Walters1978}. In other contexts, it is much more difficult to characterize the shadowing property. It therefore makes sense to ask, given a compact metric space $X$, whether a generic self-map might be \emph{expected} to have the shadowing property. In particular, given a compact metric space $X$, we would like to know if shadowing is a \emph{generic} property in $\mathcal C(X)$, the space of continuous self-maps on $X$. This question has been answered in the affirmative in a number of classes of spaces. The first results of this sort were for self-maps on the circle \cite{yano1987generic} and for more general manifolds of any dimension \cite{genericshadowingmanifolds1999}. Shadowing has also been shown to be a generic property of the space of continuous self-maps for more topologically complicated spaces, including dendrites \cite{Shadowinggenericdendrites}, chainable continua retractable to arcs \cite{bishadowgenericoprocha}, and locally connected one-dimensional continua \cite{genericshadowingmeddaugh}.  

One implication of the shadowing property is that if a map $f:X\to X$ has the shadowing property, then for each map $g$ sufficiently close to $f$ in $\mathcal C(X)$, each orbit of $g$ is reasonably approximated by an orbit of $f$ \cite{ShadowingasStructural}. In \cite{weakshadowignintroduced}, the authors explore the converse of this property, i.e. which maps $f$ have the property that for each map $g$ sufficiently close to $f$, every orbit of $f$ is reasonably approximated by an orbit for $g$. Maps satisfying this condition are said to have a form of \emph{inverse shadowing}. The authors of that paper, as well as others  \cite{Inverseshadowinggenericvarious} have refined and expanded that idea, replacing the notion of a \emph{nearby function} with the that of a \emph{$\delta$-method} coming from some predetermined class, $\mathcal T$, of methods, with functions satisfying the condition being said to have the \emph{$\mathcal T$-inverse shadowing property}. If a system has both the $\mathcal T$-inverse shadowing property and the standard shadowing property, then the system has the \emph{$\mathcal T$-bi-shadowing property}. If this property is generic in $C(X)$, then the dynamical systems on $X$ are relatively stable under perturbation. Following the same trajectory as the shadowing property, the $\mathcal T$-inverse shadowing property (for specific choices of family $\mathcal T$) has been shown to be generic in the class of self-maps for $C^\infty$-smooth manifolds \cite{weakshadowignintroduced} and in the space of self-homeomorphisms on a specific class of manifolds \cite{Inverseshadowinggenericvarious}, and hence $\mathcal T$-bi-shadowing is generic in these spaces as well. Further expanding these ideas, the authors of \cite{bishadowgenericoprocha} show that $\mathcal T$-bi-shadowing is generic for dendrites and chainable continua that can be approximated by arcs.

The main result of this paper is a significant generalization and refinement of the main result of \cite{genericshadowingmeddaugh}. In that paper, it is established that the shadowing property is generic in $\mathcal C(X)$ provided that $X$ is a locally connected continuum that can be retracted onto graphs in a sufficiently nice way, a class of continua that has now been shown to be equal to the class of locally connected one-dimensional continua \cite{peanovejnar} In this paper, we demonstrate that the hypothesis of local connectivity is not needed and that, for the specific class of $\delta$-methods called $\mathcal T_S$, the $\mathcal T_S$-inverse shadowing property, and hence also the $\mathcal T_S$-bi-shadowing property is also generic in these spaces.



\section{Preliminaries}

Let $X$ be a compact metric space and $\mathcal C(X)$ denote the set of all continuous function from $X$ to itself. Endow $\mathcal C(X)$ with the metric $\rho$ defined by $\rho(f,g)=\sup_{x\in X}d(f(x),g(x))$. A subset $P\subseteq \mathcal C(X)$ is called \emph{generic} if $P$ contains a dense $G_\delta$ subset of $\mathcal C(X)$. For $f\in \mathcal C(X)$ and $x\in X$, the \emph{orbit} of $x$ under the map $f$ is the sequence $\langle f^i(x)\rangle_{i\in \omega}.$ For $\delta>0$, a \emph{$\delta$-pseudo-orbit} of $f$ is a sequence $\langle x_i \rangle_{i\in \omega}$ such that $d(f(x_i),x_{i+1})<\delta.$ For $\varepsilon>0$, we say a point $z\in X$ \emph{$\varepsilon$-shadows} a sequence $\langle x_i \rangle_{i\in \omega}$ if $d(f^i(z),x_i)<\varepsilon$ for all $i\in\omega$.  
A mapping $f$ has the \emph{shadowing property} if for every $\varepsilon>0$ there exists $\delta>0$ such that every $\delta$-pseudo-orbit is $\varepsilon$-shadowed by some $x\in X$. For $n\in \mathbb N$, if $\varepsilon = 1/n$, let $\mathcal{P}_n(X)$ be the subset of $C(X)$ containing all mappings with the $1/n$-shadowing property and $\mathcal{P}(X)=\bigcap_{n\in \mathbb N}\mathcal{P}_n(X) $  the set of mappings with the shadowing property. 

The following definition was introduced in \cite{weakshadowignintroduced}, but we use the notation from \cite{Inverseshadowinggenericvarious}. Let $X^\omega$ be the set of all one-sided sequences in $X$. A \emph{$\delta$-method} for $f\in C(X)$ is a map $\chi :X\to X^\omega$ such that for all $x\in X$, $\chi(x)$ is a $\delta$-pseudo-orbit such that $\chi(x)_0=x$. The sequence $\chi(x)$ is called a \emph{$\delta$-pseudo-orbit produced by $\chi$}. A class $\mathcal{T}$ of $\delta$-methods is called \emph{complete} if for all $\delta>0$ there is at least one $\delta$-method $\chi\in\mathcal{T}$. Then for a complete class of $\delta$-methods $\mathcal{T}$, $f$ has the \emph{$\mathcal T$-inverse shadowing property} provided for every $\varepsilon>0$ there is some $\delta>0$ such that for any $\delta$-method $\chi\in \mathcal{T}$ and point $x\in X$, there is some $y\in X$ such that the sequence $\chi(y)$ $\varepsilon$-shadows the orbit of $x$. 
If $f$ has the shadowing property and the $\mathcal{T}$-inverse shadowing property, then we say $f$ has the \emph{$\mathcal{T}$-bi-shadowing property}.  Similar to the shadowing property, for $n\in \mathbb N$, let $\mathcal{Z}_n(X)$ be the subset of $C(X)$ containing all mappings with the $1/n$-inverse-shadowing property and $\mathcal{Z}(X)=\bigcap_{n\in \mathbb N}\mathcal{Z}_n(X) $  the set of mappings with the inverse shadowing property. 

In this paper, we consider only the class of $\delta$-methods $\mathcal{T}_S=\{\chi:$ there exists a family of continuous maps $\psi_n\in C(X)$ such that $\psi_n(\chi(x)_n)=\chi(x)_{n+1}$ for all $n\in \omega\}.$ Since there is only one class of $\delta$-methods under consideration, we will omit the $\mathcal T_S$ and refer to the $\mathcal T_S$-inverse shadowing property and the $\mathcal T_S$-bi-shadowing property as the inverse shadowing and bi-shadowing properties, respectively.


A \emph{graph} $G$ is a compact, connected space for which there exist arcs $I_1,...,I_n$ such that $G=\displaystyle\cup_{j=1}^nI_j$ and for $j\neq k$, $I_j\cap I_k$ is a subset of the set of endpoints of $I_j$. An \emph{endpoint} of $G$ is a point $x\in G$ that is an endpoint of only one such arc, while a \emph{branchpoint} is a point that is an endpoint of at least three such arcs.  If $I_j$ that does not contain any endpoints or branchpoints of G, and $I\subseteq I_j$ is an arc, then $I$ is called a \emph{free arc}. Without loss of generality, we will assume that, for $1\leq j\leq n$, each endpoint of $I_j$ is either an endpoint or a branchpoint of $G$. 

Following the terminology introduced in \cite{genericshadowingmeddaugh}, a space $X$ is called a \emph{graphoid} if for every $\varepsilon>0$, there exists a graph $G\subseteq X$ and retraction $r_G:X\to G$ such that $r_G$ is an $\varepsilon$-map (i.e. $diam(r^{-1}(x))<\varepsilon$ for each $x\in G$). The class of graphoids is an extension of the class of dendroids, as defined by Knaster in the following way: for every $\varepsilon>0$ there is a tree $T$ and retraction $r$ onto $T$ that is an $\varepsilon$-map.  It is worth pointing out that a dendroid is typically defined as an arcwise connected and hereditarily unicoherent continuum, but it is unknown if the two notions are equivalent \cite{pearl2011open}. The set of graphoids includes the topologist sine curve, the Menger curve, and the Sierpinski carpet, the harmonic fan, among others. In that paper, the author refers to locally connected graphoids as \emph{graphites}. The set of graphites was later proven to be be equal to the set of locally connected one dimensional continua \cite{peanovejnar}. However, it is not the case that every one-dimensional continuum is a graphoid. In particular, the pseudo-arc is not a graphoid as it does not contain any non-degenerate graphs. The dyadic solenoid is also not a graphoid, as the only non-degenerate graphs it contains are arcs. In \cite{genericshadowingmeddaugh}, the author showed that shadowing is generic in $C(X)$ for a graphite $X$ and proved the following lemma that will be useful for our proof that shadowing is generic for graphoids:

\begin{lemma}{(Lemma 2.2 \cite{genericshadowingmeddaugh})}\label{graphoidlemma22}
Let $X$ be a connected graphoid, and $f:X\to X$ a continuous mapping. For every $\eta >0$ there exists $\lambda>0$ such that if $G\subseteq X$ is a graph for which $r_G: X\to G$ is a $\lambda$-map, and $g:G \to X$ is in $B_\lambda(f|_G)$, then $g\circ r_G\in B_\eta(f).$
\end{lemma}

Lastly, we will need the following definitions about covers. Fix an open cover $\mathcal{U}$ of $X$. The \emph{nerve} of $\mathcal{U}$ is the simplicial complex with vertex set equal to the elements of $\mathcal{U}$ and in which the faces are the collections of element of $\mathcal{U}$ with common intersection. We say $\mathcal{U}$ is \emph{taut} if for all $U,V\in \mathcal{U}$ with $\overline{U}\cap\overline{V}\neq \emptyset$, then $U\cap V\neq \emptyset$, and no proper subset of $\mathcal{U}$ covers $X$. For $U\in \mathcal{U}$, the \emph{core} of $U$ is the portion of $U$ not contained in the closure of any other element of $\mathcal U$, i.e.\[core(U)=U\setminus \displaystyle\bigcup_{\substack{V\in \mathcal{U}\\ V\neq U}}\overline{V}.\]

The following notion will be useful for encoding information from the dynamical system using a finite open cover.

\begin{definition}
    For an open cover $\mathcal{U}=\{U_i,...,U_k\}$ of $X$ and $f\in \mathcal C(X)$, define \emph{the pattern of $f$ with respect to $\mathcal{U}$}, $\phi_{\mathcal{U},f}:\{1,...,k\}\to 2^{\{1,...,k\}} $ (or simply $\phi_f$ if $\mathcal{U}$ is unambiguous), by the following: For $i\in\{1,...,k\}$ let $\phi_{\mathcal{U},f}(i)= \{j:f(U_i)\cap U_j\neq \emptyset\}.$
\end{definition}

\section{Shadowing is Generic in $C(X)$ for Graphoids} \label{sec: Main}

Our theorem generalizes the main result in \cite{genericshadowingmeddaugh}, where the author showed that shadowing is generic for graphites. We utilize the same proof technique employed in \cite{genericshadowingmeddaugh}, as well as draw inspiration from \cite{bishadowgenericoprocha}. We will prove our theorem in a series of claims.

\begin{theorem}
    Bi-shadowing is generic in $\mathcal C(X)$ for every connected graphoid $X$. 
\end{theorem}

\begin{proof}

Let $f\in C(X)$, $n\in \mathbb N$, and $0<\varepsilon<1/n$. 

We first show there exists $g\in C(X)$ with $g\in B_\varepsilon(f)$  and open cover $\mathcal{U}$ such that we have the following properties:
    \begin{itemize}
        \item[(i)]  For all $U\in \mathcal{U}$, $diam(U)<\varepsilon$.
        \item[(ii)] If $\phi_g$ is the pattern of $g$ with respect to $\mathcal{U}$, then for every sequence $\langle j_i \rangle_{i\in \omega}$ in $\{1,2,...,k\}^\omega$ so that $j_{i+1}\in \phi_g(j_i)$, there is a point $x\in X$ such that $g^i(x)\in U_{j_i}.$
        \item[(iii)] If $j\notin\phi_g(i)$, then $\overline{g(U_i)}\cap \overline{U_j}=\emptyset.$
    \end{itemize}

To that end let $0<\alpha<\varepsilon/5$ such that $d(x,y)<\alpha$ implies $d(f(x),f(y))<\varepsilon/5$. By Lemma \ref{graphoidlemma22}, find $0<\lambda<\alpha/3$ such that if $H\subseteq X$ is a graph, $r_H:X\to H$ is a $\lambda$-map, and $h:H\to X$ is in $B_\lambda(f|_H)$, then $h\;\circ\; r_H\in B_{\alpha/3}(f).$ Also, let $\lambda$ be sufficiently small so that if $d(x,y)<\lambda$, then $d(f(x),f(y))<\alpha/3$. Lastly, we let $0<\gamma<\lambda/6$ be such that $d(x,y)<\gamma$ implies $d(f(x),f(y))<\lambda/6$.

Now let $G\subseteq X$ be a connected graph and $r:X\to G$ be a retraction such that $r$ is a $\gamma$-map. Find a taut open cover $\mathcal{V}=\{V_1,V_2,...,V_k\}$ such that for all $1\leq i\leq k$ we have $diam(V_i)<\gamma$, $V_i$ is connected, $core(V_i)=V_i\setminus\bigcup_{j\neq i}\overline{V_j}\neq \emptyset$, and the nerve of the $\mathcal{V}$ is a graph with no three cycles. Let $\mathcal{U}=\{U_1,U_2,...U_k\}$ be the open cover of $X$ defined by $U_i=r^{-1}(V_i)$ for all $0\leq i\leq k$. We claim $\mathcal{U}$ satisfies (i)-(iii). Note that for all $0\leq i\leq k$, we have that $diam(U_i)<3\gamma<\varepsilon$ by the triangle inequality, satisfying (i). Also, note that the nerve of $\mathcal{U}$ is isomorphic to that of $\mathcal{V}$, so $\mathcal{U}$ contains no three cycles and $core(U_i)\neq\emptyset.$ 

Let $F=r\circ f\circ r$. Since $r$ is a $\gamma$-map and $r\circ f|_G\in B_\gamma(f|_g)\subseteq B_\lambda(f|_G)$, by Lemma \ref{graphoidlemma22}, we have $r\circ f\circ r\in B_{\alpha/3}(f).$ Since $\alpha/3<\varepsilon/2$, we have $F\in B_\varepsilon(f).$ Also, let $K_i=\overline{F(U_i)}.$ Note that $K_i$ is a closed, connected subgraph of $G$. Let $\phi_F$ be the pattern of $F$ with respect to $\mathcal{U}$ and $\phi(i)=\{j:U_j\cap U_l\neq \emptyset\; \text{for } l\in \phi_f(i)\}$. Note that $\phi \supseteq \phi_F$. We will define the function $g$ using $F$.

Now fix $\eta <\gamma/4$ such that the following hold: 

\begin{enumerate}
    \item[(1)] If $U_i\cap U_j= \emptyset$, then $d(\overline{U_i},\overline{U_j})> 4\eta$.
    \item[(2)]For all $i\leq k,$ $\{x\in U_i:B_{4\eta}(x)\subseteq core(V_i)\}\neq \emptyset$.
    \item[(3)] For all $i,j\leq k$ with $U_i\cap U_j\neq \emptyset$, $\{x\in U_i:B_{4\eta}(x)\subseteq U_i\cap U_j\} \neq \emptyset$.
\end{enumerate}

We now show the existence of free arcs, which will serve as the foundation upon which we will construct $g$. For all $1\leq i \leq k$, define $W_i= \bigcup_{l\notin \phi(i)} B_\eta(\overline{U_l}) $ and let $C_i$ be the component of $G\setminus W_i$ containing $K_i$. Let  $G_i = \bigcup_{j\in \phi(i)} C_j $. Then $G_i$ is connected and if $j\notin \phi(i)$, $d(G_i,\overline{U_j})>\eta.$ We claim that if $j\in \phi(i)$, then $\{x\in G_i : B_\eta(x)\subseteq core(U_j)\}$ has nonempty interior in $G$. To see this, fix $1\leq i\leq k$ and suppose that $j\in \phi(i)$ such that $U_j\cap W_i=\emptyset$. Then $W_i\cap V_j=\emptyset$, so $V_j\subseteq G_i$ and we have that $\{x\in G_i : B_\eta(x)\subseteq core(U_j)\}$ has nonempty interior in $G$. Otherwise, $U_j\cap W_i\neq \emptyset.$ Let $t\in \phi(i)$ be such that $U_j\cap U_t\neq \emptyset$. Let $Q$ be a component of $G\cap U_j\setminus W_i$ that meets $G\cap \overline{U_j}\cap \overline{U_t}$. Note that $G\cap \overline{U_j}\cap\overline{U_t}\subseteq G_i$ since $G\cap U_t\subseteq G_i$. By the Boundary Bumping Theorem, we have that $Q\cap \overline{W_i}\neq \emptyset$. Then $Q$ contains points in both $\overline{U_t}$ and $\overline{W_i}$, but $B_\eta(\overline{U_s})\cap B_\eta(\overline{W_i})=\emptyset.$ Applying the triangle inequality, we have $\{x\in G_i: B_\eta(x)\subseteq core(V_j)\}$  has nonempty interior in $G$.

Now for all $1\leq i \leq k$, in order to construct $g$, we will select free arcs in each $G_i$. For each $1\leq i \leq k$ and $j\in \phi(i)$, we have that the set $\{x\in G_i: B_\eta(x)\subseteq core(U_j)\}$ has nonempty interior in $G$. Choose a collection $\{A^i_j: 1\leq i\leq k,\; j\in \phi(i)\}$ of pairwise disjoint free arcs such that $A_j^i\subseteq \{x\in G_i: B_\eta(x)\subseteq core(U_j)\} $. For indexing purposes, also choose a collection $\{A_0^i: 1\leq i\leq k \}$ of pairwise disjoint free arcs such that $A^i_0\subseteq U_i\cap G$ and $B_\eta(A_0^i)\subseteq core(U_i)$. We may assume, without loss of generality, that $\mathcal A =\{A^i_j: 1\leq i\leq k, j\in \phi(i)\}\cup  \{A_0^i: 1\leq i\leq k \}$  is pairwise disjoint. Then for each $0\leq i\leq k$ and $j\in \phi(i)\cup \{0\}$, we have that $B_\eta(A_j^i)\subseteq core(U_j)$, so $d(A_j^i,U_t)<\eta$ implies $t=j$. Also, if $x\in U_t$ and $r(x)\in A_j^i$, then $t=j.$

We will now construct $g.$ For all $A_j^i\in \mathcal{A}$, define $g_j^i:A_j^i\to G_i$ in the following way. Let $b^i_{j,0}$ and $a^i_{j,k+1}$ be the endpoints of $A_j^i$. Divide the arc $A_j^i$ into $k$ many subarcs such that for $1\leq t\leq k$ we have endpoints $b^i_{j,t}$ and $a^i_{j,t}$ ordered so that $$ b^i_{j,0}<a^i_{j,1}<b^i_{j,1}<\cdots <b^i_{j,k}<a^i_{j,k+1}. $$ Let $g_j^i$ be defined by:
\begin{enumerate}
    \item[(1)] $g_j^i(b^i_{j,0})= F(b^i_{j,0}) $
    \item[(2)] $g_j^i(a^i_{j,k+1})= F(a^i_{j,k+1}) $
    \item[(3)] For all $t\in \phi(i)$, let $g_j^i:[a^i_{j,t},b^i_{j,t} ]\to A_t^j  $ be a homeomorphism. 
    \item[(4)] For any $1\leq t,s\leq k$ upon which $(g_j^i(b^i_{j,t}),g^i_j(a^i_{j,s}))$ is undefined, let $g_j^i:[b^i_{j,t},a^i_{j,s} ]\to [g_j^i(b^i_{j,t}),g^i_j(a^i_{j,s}) ]$ be a homeomorphism onto a subarc of $G_i$ from $g_j^i(b^i_{j,t})$ to $g^i_j(a^i_{j,s})$. 
\end{enumerate}

Now let $g:X\to X$ be defined by 

$$g(x)=\begin{cases}
    g_j^i(x)\; &r(x)\in \mathcal{A}\\
    F(x) \; &r(x)\notin \mathcal{A}.
\end{cases}$$

Then $g$ is continuous by construction and if $r(x)\notin \mathcal{A}$, we have $d(g(x),f(x))=d(F(x),f(x))<\varepsilon.$ Otherwise, it remains to show that for $r(x)\in \mathcal{A}$, $d(g(x),f(x))<\varepsilon.$ To that end, fix $r(x)\in \mathcal{A}$. Then $r(x)\in A^i_j$ for some $1\leq i\leq k$ and $j\in \phi(i)\cup \{0\}$. Then $d(x,A^i_j)<\gamma$ since $r$ is a $\gamma$-map. Then $x\in B_\gamma(A^i_j)\subseteq \bigcup_{l\in\phi(i)}B_{\lambda/3}(U_l).$ Since $g(x)=g_j^i(r(x))\in g(A^i_j)\subseteq G_i\subseteq \bigcup_{l\in\phi(i)}U_l$ and $diam(U_i)<\lambda/2<\alpha/6$ for all $1\leq i\leq k$, we have $$d(g(x),f(x))\leq diam\left( \bigcup_{l\in \phi(i)} U_l \right) +2\lambda/3<diam\left( \bigcup_{l\in \phi(i)} U_l \right)+\varepsilon/2.$$ Since $d(z,y)<\lambda$ implies $d(f(z),f(y))<\varepsilon/5$, then $diam(F(U_i))<\varepsilon/5. $ For $l\in \phi(i)$, there is some $t\in \phi_F(i)$ such that $U_l\cap U_t \neq \emptyset$ and $U_t\cap f(U_i) \neq \emptyset$. Then $$diam\left( \bigcup_{l\in \phi(i)} U_l \right) \leq diam(F(U_i))+2\max \{ diam(U_l):1\leq l\leq k\}<\varepsilon/5 +\alpha/3<\varepsilon/2.$$ Then $d(f(x),g(x))<\varepsilon$, so $\rho(f,g)<\varepsilon.$

Let $\phi_g$ be the pattern of $g$ with respect to $\mathcal{U}.$ We must show $\phi=\phi_g$ in order to prove (ii) and (iii). To that end, fix $1\leq i\leq k$ and $j\in \phi(i)$. Then there is some $0\leq t\leq k$ such that $U_j\cap U_t\neq \emptyset.$ Then $g(A_j^i)\supseteq A_t^i$. By construction, $g(U_j)\cap U_i \neq \emptyset$, so $j\in \phi_g(i)$ and hence $\phi\subseteq \phi_g$. On the other hand, fix $j\in \phi_g(i).$ Find $x\in U_j$ such that $g(x)\in U_i.$ If $r(x)\in A_j^t$ for some $1\leq t\leq k$, then $t=i$, so $U_j\cap U_i\neq \emptyset$. Otherwise, $r(x) \notin \mathcal{A},$ so $g(x)=F(x)$. In either case, we have $j\in \phi(i)$, so $\phi(i)=\phi_g(i).$ Hence, $\phi=\phi_g.$ 

To prove $g$ satisfies (ii), let $\langle j_i\rangle_{i\in \omega}$ be a sequence in $\{1,...,k\}^\omega$ such that $j_{i+1}\in \phi_g(j_i) $ for all $i\in \omega.$ Let $j_{-1}=0. $ For each $i\in \omega$, by construction, we have $g(A_{j_{i+1}}^{j_{i}})\supseteq A_{j_{i}}^{j_{i+1}}$. By compactness, we may find $x\in\displaystyle \bigcap_{i\in \omega} g^{-i}(A^{j_i}_{j_{i-1}})\neq \emptyset.$ Then for all $i\in \omega,$ $g^i(x) \in A^{j_i}_{j_{i-1}}\subseteq U_{j_i}.  $ 

To prove (iii), we prove the contrapositive. Suppose $x\in \overline{g(U_i)}\cap \overline{U_j}.$ Then there is a sequence $g(z_n)\to x$ in $g(U_i)$ and a sequence $y_n\to x$ in $U_j.$ Let $\beta= diam(U_j)$. Find $N$ such that $n\geq N$ implies $d(g(x_n),x),\; d(y_n,x)<\beta/3$. Then $d(y_N,g(x_N))<\beta$, so $g(x_N)\in U_j$. Therefore, $g(U_i)\cap U_j\neq \emptyset$, so $j\in \phi_g(i).$ Therefore, the mapping $g$ and cover $\mathcal{U}$ together satisfy (i)-(iii).

We now claim that there exists $\beta>0$ such that for all $h\in C(X)$ with $\rho(h,g)<\beta$, if $\phi_h$ is the pattern of $h$ with respect to $\mathcal{U}$, we have $\phi_h=\phi_g$, and for all $\langle j_i\rangle_{i\in \omega} $ in $\{1,...,k\}^\omega$ with $j_{i+1}\in \phi_h(j_i)$, there exists $x\in X$ such that $h^i(x)\in U_{j_i}$.

Note that by (iii), for $1\leq i\leq k$, $\overline{g(U_i)}\cap \overline{U_j}\neq \emptyset$ if and only if $j\in \phi_g(i)$. Then for all $1\leq i\leq k$ we may find $\eta>\tau_i>0$ such that $d(\overline{g(U_i)},\overline{U_j})<\tau_i$ if any only if $j\in \phi_g(i).$ Let $\tau=\displaystyle\min_{0\leq i\leq k}\{\tau_i\}$. If $\rho(h,g)<\tau$, then $h(\overline{U_i})\subseteq \bigcup_{j\in \phi_g(i)}\overline{U_j}\setminus W_i.$ Therefore, $\phi_h(i)\subseteq \phi_g(i)$, so $\phi_h\subseteq \phi_g$. 

To show $\phi_g\subseteq \phi_h$, let $\tau>\xi >0$ such that for all $1\leq i,j\leq k$ such that $A^i_j\in \mathcal{A}$, we have $B_\xi(A^i_j)\cap G$ is an arc such that $\xi$ is strictly less than the distance from $[a_{j,1}^i, b_{j,k}^i]$ to either endpoint of $A_j^i$. Then, if $a,b\in G$ such that, without loss of generality, $d(b, b_{j,0}^i)<\xi$ and $d(a,a_{j,k+1}^i )<\xi$, then the arc from $b$ to $a$ is contained in $B_\xi(A_j^i)\cap G$ and contains $[a_{j,1}^i, b_{j,k}^i]$. Choose $\xi>\beta>0$ such that if $d(a,b)<\beta$, then both $d(r(x),r(y))<\xi$ and $d(g(x),g(y))<\xi$. We will show that $\beta$ satisfies the claim. 

Fix $h\in C(X)$ with $\rho(h,g)<\beta$. Then $\phi_h\subseteq\phi_g$ since $\beta<\tau$. To show $\phi_g\subseteq\phi_h$, fix $1\leq i\leq k$ and find $j$ such that $A_j^i\subseteq U_j$. Fix $t\in \phi_g(i)$ and $x\in [b_{j,t}^i,a_{j,t}^i].$ Then $h(x) \in B_\beta(g(x))\subseteq B_\eta(A_t^j)\subseteq core(U_j)$, so $j\in \phi_h(i)$ and we have $\phi_g=\phi_h$.

We will prove the second part of our claim by induction. Fix $\langle j_i\rangle_{i\in \omega} $ in $\{1,...,k\}^\omega$ with $j_{i+1}\in \phi_h(j_i)$. Let $I_0= [a_{0,j_1}^{j_0}, b_{0,j_1}^{j_0}] \subseteq A_0^{j_0}.$ Since $\rho(h,g)<\beta$, then $$d(r(h(a_{0,j_1}^{j_0})), r(g(a_{0,j_1}^{j_0})))  <\xi.$$  By the definition of $g$, we have that $ r(g(a_{0,j_1}^{j_0}))=b_{j_0,0}^{j_1}$, so $d(h(a_{0,j_1}^{j_0}),b_{j_0,0}^{j_1})<\xi.$ A similar argument shows that $d(h(b_{0,j_1}^{j_0}), a^{j_1}_{j_0,k+1})<\xi.$ Therefore, we have that the arc from $h(a_{0,j_1}^{j_0})$ to $h(b_{0,j_1}^{j_0})$ is completely contained in $B_\xi (A^{j_1}_{j_0})\cap G$ and also $r(h(I_0))\supseteq [b_{j_0,1}^{j_1},a^{j_1}_{j_0,k}]$. Therefore, since $r\circ h$ is continuous, $r(h(I_0))$ is connected and we may find some $I_1\subseteq I_0\subseteq U_{j_0}$ such that $r(h(I_1))= [b_{j_0,1}^{j_1},a^{j_1}_{j_0,k}]$.

Now suppose there exists $N\in \mathbb{N}$ such that $I_n$ is a subarc of $I_{n-1}$ such that $h^n(I_n)\subseteq B_\beta(A_{j_{n-1}}^{j_n})$, $r(h^n(I_n))\subseteq [b_{j_{n-1,1}}^{j_n},a_{j_{n-1,k}}^{j_n}]$, and $h^{n-1}(I_n)\subseteq U_{j_n-1}$ for all $n<N$. By the continuity of $r\circ h$, fix $a,b\in I_n$ so that $r(h^n(a))= a^{j_n}_{j_{n-1}, j_{n+1}} $ and $r(h^n(b))= b^{j_n}_{j_{n-1}, j_{n+1}} $ so that $r(h^n([a,b]))= [a^{j_n}_{j_{n-1}, j_{n+1}} ,b^{j_n}_{j_{n-1}, j_{n+1}} ]$. Then for all $y\in [a,b] $, $d(h^{n+1}(y),g(h^n(y)))<\beta$ and $g(h^n(x))\in g([a^{j_n}_{j_{n-1}, j_{n+1}} ,b^{j_n}_{j_{n-1}, j_{n+1}} ])= A^{j_{n+1}}_{j_n}.$ Therefore, $d(h^{n+1}(y),A^{j_{n+1}}_{j_{n}})<\beta $, so $h^{n+1}([a,b])\subseteq B_\beta(A^{j_{n+1}}_{j_n}).$

Since $r$ is a $\gamma$-map and a retraction, we have that $h^n([a,b])\subseteq B_\lambda(A^{j_n},{j_{n-1}})\subseteq U_{j_n}$. Note that $g(h^n(a))= g(r(h^n(a))= g(a^{j_n}_{j_{n-1}, j_{n+1}})= b^{j_{n+1}}_{j_{n},0}.$ Then we have $d(h^{n+1}(a), b^{j_{n+1}}_{j_{n},0})<\beta$ and $r(b^{j_{n+1}}_{j_{n},0})=b^{j_{n+1}}_{j_{n},0}$, so $d(r(h^{n+1}(a)),r(b^{j_{n+1}}_{j_{n},0}))<\xi$. By a similar argument as before, we have $d(r(h^{n+1}(b), a^{j_{n+1}}_{j_{n},k+1})<\xi.$ 

Therefore, we have that the arc from $h^{n+1}(a)$ to $h^{n+1}(b)$ is completely contained in $B_\xi (A^{j_{n+1}}_{j_n})\cap G$ and also $r(h(I_n))\supseteq [b^{j_{n+1}}_{j_{n},0},a^{j_{n+1}}_{j_{n},k+1}]$. Therefore, since $r\circ h$ is continuous, $r(h(I_n))$ is connected and we may find some $I_{n+1}\subseteq I_n\subseteq U_{j_n}$ such that $r(h(I_{n+1}))= [b^{j_{n+1}}_{j_{n},0},a^{j_{n+1}}_{j_{n},k+1}]$. Since $I_{n+1}\subseteq[a,b]$, we have $h^{n+1}(I_{n+1})\subseteq B_\gamma(A^{j_{n+1}}_{j_n})$ and $h^n(I_{n+1})\subseteq U_{j_n}.$

Therefore, if $x\in I_n$, $h^i(x)\in U_{j_i}$ for all $0\leq i\leq n$. By compactness , we may find $z\in \bigcap_{n\in \omega}I_n \neq \emptyset$, and $h^i(z)\in U_{j_i} $ for all $i\in \omega$, which proves our claim.

We may now show $B_\beta(g)\subseteq \mathcal{P}_n(X)$.

Let $h\in B_\beta(g)$ and choose $0<\delta<\alpha/3$ to be the Lebesgue number for $\mathcal{U}$. Let $\langle x_i \rangle _{i\in \omega}$ be a $\delta$-pseudo-orbit in $(X,f\circ r)$. Then for all $i\in \omega$, it follows that $diam(\{x_i, f(r(x_{i-1})) \})<\delta<\alpha/3.$ Fix $j_i\leq k$ so that $x_i,f(r(x_{i-1}))\in U_{j_i}$ and $j_0$ so that $x_0\in U_{j_0}$. Then for all $i \in \omega$, $j_{i+1}\in \phi_h(j_i)$, so by the previous claim, there exists $x\in X$ such that for all $i\in \omega$ we have $h^i(x)\in U_{j_i}.$ Then $x_i,h^i(x)\in U_{j_i}$, so $d(x_i,h^i(x))<diam(U_{j_i})<\alpha/3<\varepsilon<1/n,$ so $x$ $1/n$-shadows $\langle x_i \rangle _{i\in \omega}$. Therefore, $h\in \mathcal{P}_n(X).$ Hence, $B_\beta(g)\subseteq \mathcal{P}_n(X)$.

Since $n\in \mathbb N$ was arbitrary, by our above claims, we have that for all $\varepsilon>0$, there is some $n\in \mathbb N$, $0<\beta_n<\varepsilon$, and $g\in B_\varepsilon(f)$ such that $ B_{\beta_n}(g)\subseteq \mathcal{P}_n(X)$. Therefore $\mathcal{P}_n(X)$ contains a dense open subset of $C(X)$ and $\mathcal{P}(X)\supseteq \bigcap_{n\in \mathbb N}\mathcal{P}_n(X). $ Hence, $\mathcal{P}(X)$ contains a dense $G_\delta$ subset of $C(X)$.

It remains to show that $\mathcal{Z}(X) $ contains a dense $G_\delta$ subset of $C(X).$ To show this, let $\delta>0$ be the Lebesgue number for $\mathcal{U}$, as before and $h\in B_\beta(g)$. We claim that for every $\delta$-pseudo-orbit in $(X,h)$, $\langle x_i\rangle_{i\in \omega}$, and every $\delta$-method $\chi\in\mathcal{T}_{S,h,\delta}, $ we have $\langle x_i\rangle_{i\in \omega}$ is $1/n$-shadowed by some pseudo-orbit produced by $\chi$ and some orbit of $h$.

Let $\chi\in\mathcal{T}_{S,h,\delta}$ be a $\delta$-method for $h$. Let $\psi_i\in C(X)$ be a sequence in $C(X)$ such that for all $i\in \omega$, $\rho(h,\psi_i)<\delta$ and $\chi(z)_{i+1}= \psi_i(\chi(z)_i)$ for all $z\in X$. Let $\psi_{-1}=id_X.$ Now let $\langle x_i\rangle_{i\in \omega}$ be a $\delta$-pseudo-orbit for $h$. By our previous claim, we have there exists $x\in X$ such that $h^i(x) \in U_{j_i}$ for all $i\in \omega$, so $h$ has the shadowing property. Let $\chi(x) = \langle (\psi_{i-1}\circ\psi_{i-2}\circ\cdots \circ \psi_{-1})(x)\rangle_{i\in\omega} $ be a $\delta$-pseudo-orbit produced by $\chi.$ Then $d(\chi(x)_i,x_i)\leq d(x_i,h^i(x))+d(h^i(x),(\psi_{i-1}\circ\cdots \circ \psi_{-1})(x))<\varepsilon/3+\delta<\varepsilon$, so $h$ has $\mathcal{T}_S$-inverse shadowing and, hence, $\mathcal{T}_S$-bi-shadowing.

    We will now show that $\mathcal{Z}$ contains a dense $G_\delta$ subset of $\mathcal C(X)$ Since $\varepsilon>0$ was arbitrary, we have for all $\varepsilon>0$ there is some $0<\beta_N<\varepsilon$ and $\delta>0$ such that for $h\in B_{\beta_N}(f)$ we have for every $\delta$-pseudo-orbit, $\langle x_i\rangle_{i\in \omega}$, and every $\delta$-method $\chi\in\mathcal{T}_{S}, $ we have $\langle x_i\rangle_{i\in \omega}$ is $1/N$-shadowed by some pseudo-orbit produced by $\chi$ and some orbit of $h$. Hence $h\in \mathcal{Z}_N(X)$. Therefore, for all $n\in \mathbb N$ $,\mathcal{Z}_n(X)$ contains a dense open subset of $\mathcal{Z}.$ Then $\mathcal{Z}(X)\supseteq \bigcap_{n\in \mathbb N}\mathcal{Z}_n(X).$ Hence, $\mathcal{Z}(X) $ contains a dense $G_\delta$ in $C(X).$
\end{proof}

\section{Conclusion}
The results in Section \ref{sec: Main} concern only the subspace $\mathcal C(X)$ of continuous self-maps from $X$ to $X$. Previous work on the question of genericity also consider the subspaces $\mathcal S(X)$ and $\mathcal H(X)$, consisting of continuous self-surjections and continuous self-homemorphisms, respectively. In our construction, the mappings are almost certainly not surjective. The methods used in \cite{genericshadowingmeddaugh} to establish that shadowing is generic in $S(X)$, where $X$ is a graphite, rely heavily on the local connectivity of graphites, and so the argument cannot be generalized. Concerning $\mathcal H(X)$, the authors of \cite{Inverseshadowinggenericvarious} show that $\mathcal{T}_S$-inverse shadowing is generic in $\mathcal H(X)$ for a specific class of manifolds, as well as a for a different class of $\delta$-methods, but their arguments depend heavily on the relative homogeneity of manifolds. Indeed, it is not surprising that addressing genericity of shadowing in either $\mathcal S(X)$ or $\mathcal H(X)$ for a graphoid (or other non-locally connected space) is difficult due to the local inhomogeneities of these spaces. This leads is to conclude with the following questions.

\begin{question}
    Is there a space $X$ for which shadowing is generic in $S(X)$ that is not locally connected?
\end{question}

\begin{question}
    Is there a space $X$ for which $\mathcal{T}_S$-bi-shadowing is generic in $H(X)$ that is not locally connected?
\end{question}

In this paper, and indeed in much of the literature concerning genericity of shadowing and inverse shadowing, it is the case that both are generic properties \cite{Inverseshadowinggenericvarious, genericshadowingmeddaugh, yano1987generic}.

\begin{question}
    Does there exist a space $X$ and class of $\delta$-methods $\mathcal{Q}$ in which $\mathcal{Q}$-inverse shadowing is generic in $C(X)$, but shadowing is not generic, or vice versa?
\end{question}

\bibliographystyle{plain}
\bibliography{Bibliography}

\end{document}